\documentclass[bj]{imsart}

\RequirePackage{amsthm,amsmath,amsfonts,amssymb}
\RequirePackage[numbers]{natbib}
\usepackage{amsmath}
\usepackage{amssymb}
\usepackage{amsfonts}
\usepackage{bbm}
\usepackage{comment}
\usepackage{array}
\usepackage{graphicx} 
\usepackage[dvipsnames]{xcolor}
\usepackage{caption}
\usepackage{makecell}
\usepackage{array}
\newcolumntype{?}{!{\vrule width 2.5pt}}

\startlocaldefs
\theoremstyle{plain}
\newtheorem{theorem}{Theorem}
\newtheorem{corollary}{Corollary}
\newtheorem{lemma}{Lemma}
\newtheorem{definition}{Definition}

\endlocaldefs

\begin{document}

\begin{frontmatter}
\title{Stochastic Sandpile on a Cycle}
\runtitle{Stochastic Sandpile on a Cycle}

\begin{aug}
\author[A]{\inits{AM}\fnms{Andrew} \snm{Melchionna}\ead[label=e1]{am2433@cornell.edu}},
\address[A]{1. \printead{e1}}
\end{aug}

\begin{abstract}
In the stochastic sandpile model on a graph, particles interact pairwise as follows: if two particles occupy the same vertex, they must each take an independent random walk step with some probability $0<p<1$ of not moving. These interactions continue until each site has no more than one particle on it. We provide a formal coupling between the stochastic sandpile and the activated random walk models, and we use the coupling to show that for the stochastic sandpile with $n$ particles on the cycle graph $\mathbb{Z}_n,$ the system stabilizes in $O(n^3)$ time for all initial particle configurations, provided that $p(n)$ tends to $1$ sufficiently rapidly as $n \rightarrow \infty$.
\end{abstract}

\begin{keyword}
\kwd{Interacting particle systems}
\kwd{Coupled Markov chains}
\kwd{Self-organized criticality}
\end{keyword}

\end{frontmatter}


\section{Introduction}
Sandpile models were first introduced by Bak, Tang and Wiesenfield in the 1980's \citep{r1}. While sandpile models come in many different forms (e.g. \citep{r2, r3, r4}), they can all be described as a system of interacting particles which live on the vertices of a graph, and move (deterministically or randomly) along the edges of the graph until the particles are sufficiently well-dispersed. This paper is concerned primarily with the so-called stochastic sandpile model (abbreviated SS) and its dynamics on the cycle graph, particularly with the number of stochastic updates needed until a stable state has been reached. We present an overview of the model here, giving the details in Section 2. 

Begin with $n$ indistinguishable particles on the vertices of $\mathbb{Z}_n$ in an arbitrary configuration. The locations of the particles are updated in discrete time steps as follows. If at any time a site has two or more particles on it, the site is deemed \textit{unstable} and must be \textit{toppled}. In toppling the unstable site, we let two particles at the site each perform an independent, lazy, symmetric random walk on the graph. That is, for \textit{lazy parameter} $0<p(n)<1$, a particle does not move with probability $p(n),$ moves one step clockwise with probabilty $\frac{1-p(n)}{2},$ and moves one step counterclockwise with probability $\frac{1-p(n)}{2}.$ Both particles move independently according to this distribution. This process continues, toppling one unstable site per timestep, until each site has exactly one particle, at which time the system is deemed \textit{SS-stable}. Define the \textit{odometer function for SS} to be the function $v: \mathbb{Z}_n \rightarrow \mathbb{N}\cup \{0\}$ which counts the number of times each site topples in the stabilization process. The stochastic sandpile model enjoys the so-called abelian property, which states that given an arbitrary initial particle configuration and quenched randomness, the odometer function is independent of the choice of unstable site that we topple at each time step, making $v$ a well-defined random function (see Section 2.3). 

Our main result will compare SS to a similar system: the activated random walk model (abbreviated ARW) (see \citep{r6}). We briefly describe ARW as a continuous-time process here, and give a discrete-time construction in the Section 2. Begin with $n$ particles on $\mathbb{Z}_n$. A particle in the ARW model can be in one of two states, "active" or "sleeping". Active particles perform continuous-time symmetric random walk along the graph, dictated by a Poisson clock with jump rate 1. An active particle also carries a second, independent Poisson clock with rate $\lambda$ whose ticks tell the particle to try to fall asleep. If the particle is alone on its vertex, it successfully changes its state to "sleeping". If it shares its vertex with other particles, it fails to sleep and remains in its "active" state. Sleeping particles can only exist on a vertex where they are the only particle, and remain asleep until another particle visits their vertex, at which time the particle instantaneously changes to an active state. This process continues until the configuration is \textit{ARW-stable}, that is, until there is exactly one sleeping particle at every vertex (and no active particles). We define the \textit{odometer function for ARW} to be the function $u: \mathbb{Z}_n \rightarrow \mathbb{N} \cup \{0\}$ which counts the number of times a Poisson clock (including the jump clocks and the sleep clocks) ticked at each site $x$ during the stabilization process. 

Let $\eta$ represent a state of ARW (i.e. the number particles in each state are on each site of $\mathbb{Z}_n$), and let $|\eta|$ represent the corresponding SS state with the same number of particles on each site in $\eta$, ignorant of whether the particles are active or sleeping. Insofar as we are concerned with ARW as it relates to the SS dynamics, we define the random time $T_{-1}$ to be the first time that the ARW model has one particle (regardless of state) on each site in $\mathbb{Z}_n$, and we define $\overline{u}$ to be the odometer function immediately after $T_{-1}$ (see section 2 for a rigorous formulation).

\begin{theorem}
Consider the ARW model with sleep rate $\lambda$ starting from arbitrary $n$-particle state $\eta_0,$ and let $u$ and $\overline{u}$ represent the odometer functions for the full stabilization process and stabilization process stopped at time $T_{-1}.$ Similarly, let $v$ represent the odometer function for SS with lazy parameter $p = \frac{\lambda}{1+\lambda}$ and starting with state $|\eta_0|.$ Then for all $x\in \mathbb{Z}_n$,
\[
\lceil \frac{\overline{u}(x)}{2} \rceil \overset{d}{=} v(x)
\]

and $\overline{u}(x) \leq u(x)$ almost surely for all $x \in \mathbb{Z}_n$.
\end{theorem}
By combining Theorem 1 with Lemma 3, it follows that that for $p(n)$ tending to $1$ sufficiently rapidly, the stochastic sandpile on the $n$-cycle stabilizes in time $O(n^3)$ when starting from an arbitrary $n$-particle configuration.
\begin{corollary} Consider the stochastic sandpile model starting from an arbitrary configuration of $n$ particles on $\mathbb{Z}_n,$ and let $0 < p(n) < 1$ be a function of $n$ satisfying 
\[
\limsup_{n \rightarrow \infty} \,[\log(n) \cdot \big(1-p(n)\big) ]< 2.
\]
Then
\[
\lim_{n \rightarrow \infty} \mathbbm{P} (T > A n^3) = 0,
\]
where $T$ is the total time to stabilization, and $A$ is a constant independent of $n$. \end{corollary}

Sandpiles serve as models of self-organized criticality \citep{r9}, which is the tendency of certain physical systems to gravitate towards critical states without the tuning of any parameters. An important signature of criticality found in these models is that for sandpiles on infinite lattices, avalanche sizes follow a power-law distribution \citep{r10}, where an avalanche is the propagation of waves of high particle density through regions of low particle density. Lack of control over the effects of these avalanches remains an obstacle to proofs of fast stabilization of the stochastic sandpile model. For example, say the n-cycle contains a region of stability, that is, an interval along which each site contains exactly one particle. If a particle enters that interval and causes one of the sites to become unstable, large chain reactions can occur in which many of the sites in the interval must topple, leaving disastrously unpredictable particle configurations in its wake.

As we develop below, the ARW model couples nicely to the SS model, has dynamics on the cycle graph which are well-studied \citep{r5, r7}, and can alleviate some of the aforementioned difficulties in analyzing SS. Aside from its own mathematical interest, the ARW model is useful to proofs of fast stabilization of the stochastic sandpile model for two main reasons. Firstly, the ARW model provides more paths to stability, in the sense that particles move independently rather than in pairs. This flexibility can be exploited by choosing the combinations of particle moves which are simplest and best-suited for a given analysis. The second important facet of ARW compared to SS is that its stabilization takes longer, since all particles must be asleep in addition to being alone on a vertex in order to be stable. This inequality between the stabilization times of the two models is useful to the goal of upper-bounding the stabilization time of SS. The proof of the Corollary 1 uses an analogous result proven in \citep{r7}, which shows that ARW on the $n$-cycle undergoes a phase transition: for $\lambda(n)$ growing sufficiently rapidly with $n \rightarrow \infty$, stabilization occurs in $O(n^3)$ time, with stabilization time exponential in $n$ otherwise. We state the fast-phase result below in Lemma 3. 

\begin{figure}
\label{SSPlots} 
 \centering
\captionsetup{singlelinecheck=off}
\includegraphics[width = 4in]{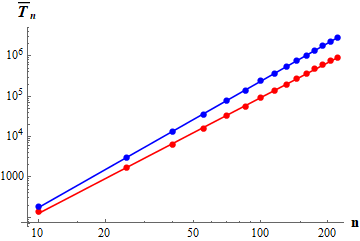}
\caption{\textbf{SS simulations:} A log-log plot of average time to stabilization $\overline{T}_n$ over 200 trials for various values of $n$, with linear fits. The colors of the plots correspond to lazy parameters $\textcolor{blue}{p = \frac{1}{2}}$ and $ \textcolor{red}{p = \frac{\log(n)}{1 + \log(n)}}$. The slope of the blue line is $\approx \textcolor{blue}{3.14}$ and the slope of the red line is $\approx \textcolor{red}{2.88}$.}
\end{figure}

An outstanding open problem, whose investigation led to this result, is to prove that SS on the $n$-cycle with \textit{constant} (i.e. independent of $n$) lazy parameter $p$ stabilizes in polynomial time. Simulations suggest that the system does indeed stabilize in time $O(n^3).$ In Figure 1, we provide a log-log plot of average stabilization time vs. $n$. The result appears linear with slopes close to $3$, giving evidence for $O(n^3)$ stabilization time of the constant-laziness SS model. The heuristic for fact that fast stabilization of the SS model with large lazy parameter appears easier to prove than that of the constant laziness SS model is as follows: despite there being more null topplings (topplings with no particles being displaced) in the lazy version, there are fewer moves resulting in both particles being displaced, which is the mechanism driving the unpredictability of avalanches.

Some previous investigations have focused on phase transitions of SS of infinite lattices. In \citep{r4}, the authors consider SS on $\mathbb{Z},$ with initial particle configuration given by i.i.d. Poisson random variables with mean $\mu$ at each site. They show the existence of a critical density $\mu_c$ such that for $\mu < \mu_c$ the system will eventually stabilize, but for $\mu > \mu_c,$ the system does not stabilize. On finite graphs with sufficiently few particles, the system stabilizes in finite time almost surely, so the question of how long it the system will take to stabilize seems like a natural one. It is worth noting that for any more than $n$ particles on a finite graph with $n$ vertices, the system cannot stabilize, since there are no configurations with at most one particle on each site. Thus, this paper investigates the stabilization properties of a system with the maximum number of particles.

The remainder of the paper is structured as follows. In Section 2, we rigorously introduce both models and some of their important properties. In Section 3, we develop some Markov chain machinery that allows us to couple the two models, with the goal of expressing the SS Markov chain as a "quotient" of ARW Markov chain. In Section 4 we develop the coupling and give the proof of Theorem 1.

\section{Preliminaries}
\subsection{The Stochastic Sandpile Model on $\mathbb{Z}_n$}
We label the sites in $\mathbb{Z}_n$ as $\{0,1,2,..., n-1\},$ with indices increasing as we go counterclockwise. We make the following definition: 
\begin{definition}
A \textbf{stochastic sandpile configuration} is a function $s: \mathbb{Z}_n \rightarrow \mathbb{N} \cup \{0\}$ satisfying
\begin{equation}
\sum_{x \in \mathbb{Z}_n} s(x) = n,
\end{equation}
and we let $S$ refer to the set of all stochastic sandpile configurations.
\end{definition}
We think of $s(x)$ as representing the number of particles on site $x \in \mathbb{Z}_n$, with equation (1) enforcing that there must always be exactly $n$ total particles on $\mathbb{Z}_n.$

For a given configuration $s$, a site $x$ is called \textit{unstable for} $s$ if $s(x) \geq 2.$ If such a site exists, we call $s$ unstable, otherwise, we call $s$ stable.

\begin{table}
\begin{center}
{\renewcommand{\arraystretch}{.001}
\begin{tabular}{ | m{1cm}  | m{10em} | m{1.8cm} | }
  \Xhline{5\arrayrulewidth}
\vspace{.1in}
  $\tau^x$ \vspace{.1in} & Transformation of $s$ & Probability 
\\ 
    \Xhline{5\arrayrulewidth}
 $\tau^x_{(0,0)}$ &\begin{center} None \end{center} & $p^2$  \\ 
  \Xhline{1\arrayrulewidth}
   $\tau^x_{(1,0)}$ & {\begin{align*} s(x) & \rightarrow s(x) - 1 \\ s(x-1) & \rightarrow s(x-1) + 1\end{align*}} & $p (1-p)$  \\
    \Xhline{1\arrayrulewidth}
   $\tau^x_{(0,1)}$ & {\begin{align*} s(x) & \rightarrow s(x) - 1 \\ s(x+1) & \rightarrow s(x+1) + 1\end{align*}} &  $p (1-p)$ \\
   \Xhline{1\arrayrulewidth}
   $\tau^x_{(1,1)}$ & {\begin{align*} s(x) & \rightarrow s(x) - 2 \\ s(x+1) & \rightarrow s(x+1) + 1 \\ s(x-1) & \rightarrow s(x-1) + 1\end{align*}} &  $\frac{ (1-p)^2}{2}$ \\
    \Xhline{1\arrayrulewidth}
   $\tau^x_{(2,0)}$& {\begin{align*} s(x) & \rightarrow s(x) - 2 \\ s(x-1) & \rightarrow s(x-1) + 2\end{align*}} & $\frac{ (1-p)^2}{4}$  \\
  \Xhline{1\arrayrulewidth}
   $\tau^x_{(0,2)}$& {\begin{align*} s(x) & \rightarrow s(x) - 2 \\ s(x+1) & \rightarrow s(x+1) + 2\end{align*}}& $\frac{ (1-p)^2}{4}$  \\
    \Xhline{5\arrayrulewidth}
\end{tabular}
}
\end{center}
\caption{\label{tab:table-name}All values of the SS toppling operator $\tau^x$ with corresponding effects on the configuration and probabilities of occurence.}
\end{table}

The stochastic sandpile model begins with an initial particle configuration $s_0$. We stabilize the sandpile in discrete time steps as follows. Let $s_t$ represent the particle configuration at time $t \in \mathbb{N} \cup \{0\}$.We choose a site $x_t$ which is unstable for $s_t$ and \textit{topple} this site by letting two particles perform independent lazy symmetric random walk steps with lazy parameter $0 < p(n) < 1$. That is, two particles at site $x_t$ each relocate independently according to the following distribution: do not move with probability $p(n),$ step clockwise with probability $\frac{1-p(n)}{2},$ and step counterclockwise with probability $\frac{1-p(n)}{2}.$ Henceforth, we leave implicit the dependence of $p$ on $n$ for ease of notation. We summarize the results of a stochastic sandpile toppling with the pair of nonnegative integers $(\rho_-,\rho_+),$ where $\rho_-$ and $\rho_+$ represent the number of particles that stepped clockwise and counterclockwise, respectively, in a toppling. Note that $(\rho_-, \rho_+)$ must satisfy the following two conditions:
\begin{align}
\rho_-, \rho_+ &\geq 0 \\
\rho_- + \rho_+ &\leq 2 .
\end{align}

We let $\tau^x$ be the (random) operator which represents a toppling at site $x,$ with $\tau^x$ taking values $\tau^x_{(\rho_-,\rho_+)}$. We denote the resulting configuration by $s_{t+1} = \tau^{x_t} s_t$. Table 1 summarizes the possible values of $\tau^x$ and their effects on the configuration $s$.

\subsection{The Activated Random Walk Model on $\mathbb{Z}_n$}
In the activated random walk model, particles are in one of two states, "active" or "sleeping". We label the sites in $\mathbb{Z}_n$ in the same manner as above. Let $\mathbb{N}_0 := \mathbb{N} \cup \{0\},$ and define the ordered set $\mathbb{N}_{\mathfrak{s}} = \mathbb{N}_0 \cup \{\mathfrak{s}\},$ with $0 < \mathfrak{s} < 1 < 2 < ...\,$. We define an activated random walk particle configuration which describes the number and type of particles at each site.

\begin{definition}
An \textbf{activated random walk configuration} is a function $\eta: \mathbb{Z}_n \rightarrow \mathbb{N}_{\mathfrak{s}} \cup \{0\}$ with
\begin{equation}
\sum_{x \in \mathbb{Z}_n} |\eta(x)| = n,
\end{equation}
where we let $|\mathfrak{s}| := 1.$ We let $H$ refer to the set of all activated random walk configurations. 
\end{definition}
$\eta(x) = \mathfrak{s}$ indicates that there is a single sleeping particle on site $x$, $\eta(x) \in \mathbb{N}$ indicates that there are $\eta(x)$ active particles on $x$, and $\eta(x) = 0$ indicates that there are no particles on $x$. Equation (2) enforces that there are $n$ total particles on $\mathbb{Z}_n.$

For a given configuration $\eta,$ we call a site $x$ \textit{unstable} if $\eta(x) \geq 1.$ A configuration with at least one unstable site is deemed unstable, otherwise it is stable. Note that $\eta$ is stable if and only if $\eta(x) = \mathfrak{s}$ for all $x$, that is, if each site contains one sleeping particle and no particles are active.

We describe here the discrete-time formulation of the model, whose dynamics are equivalent with the continuous-time formulation given in the introduction. Our activated random walk model will begin with an arbitrary initial $n$-particle configuration $\eta_0$. We stabilize in discrete time steps as follows. Let $\eta_t$ represent the particle configuration at time $t \in \mathbb{N} \cup \{0\}$.We choose a site $x_t$ which is unstable for $\eta_t$ and \textit{topple} this site by letting an active particle try to fall asleep with probability $\frac{\lambda}{1+ \lambda}$, and otherwise taking a (non-lazy) symmetric random walk step with probability $\frac{1}{1+\lambda}.$ Note that an ARW toppling involves the possible displacement of one particle, while an SS toppling involves the possible displacements of two particles. 

We now describe the effects of each toppling move on the configuration $\eta$. Firstly, sleeping particles can only exist on a site by themselves. So, if we are toppling an active particle at a site $x$ where $\eta(x) \geq 2$ and that particle tries to go to sleep, it fails and remains active (remaining at site $x$). If it is alone (i.e. $\eta(x) = 1$), then it successfully goes to sleep (so that $\eta(x)$ becomes $\mathfrak{s}$). Secondly, if an active particle takes a step and arrives at a site containing a sleeping particle, the sleeping particle immediately wakes up, resulting in two active particles. 

We let $\delta^x$ be the operator which represents a toppling at site $x,$ with $\delta^x$ taking values $\delta^x_{-1},\delta^x_{+1},\delta^x_{\mathfrak{s}}$, which represent a particle taking a step clockwise, counterclockwise, and attempting to fall asleep, respectively. We denote the resulting configuration by $\eta_{t+1} = \delta^{x_t} \eta_t$. 

We formalize the above by defining the following operations on $\mathbb{N}_{\mathfrak{s}}:$
\begin{align}
1\cdot\mathfrak{s} &= \mathfrak{s} &\\
k\cdot \mathfrak{s} &= k & k \geq 2  \\
k + \mathfrak{s} &= k + 1  &k \geq 1 \\
k - \ell &= k - \ell  & k \geq \ell &\geq 0,  \quad k, \ell \in \mathbb{N}_{0}
\end{align}
and summarizing the distribution of the results of an ARW toppling at site $x$ in Table 2.
\begin{table}
\begin{center}
{\renewcommand{\arraystretch}{.001}
\begin{tabular}{ ? m{1cm}  | m{10em} | m{1.8cm} ? }
    \Xhline{5\arrayrulewidth}
\vspace{.1in}
   $\delta^x$ \vspace{.1in} & Transformation of $\eta$ & Probability \\ 
    \Xhline{5\arrayrulewidth}
\vspace{.1in}
$\delta^x_{\mathfrak{s}}$ \vspace{.1in}&\begin{center}$\eta(x) \rightarrow \eta(x) \cdot \mathfrak{s} $ \end{center}& $\frac{\lambda}{1 + \lambda}$ 

  \\ 
  \Xhline{1\arrayrulewidth}
  $\delta^x_{+1}$&{\begin{align*} \eta(x) & \rightarrow \eta(x) - 1 \\ \eta(x+1) & \rightarrow \eta(x+1) + 1 \end{align*}} & $\frac{\lambda}{2(1 + \lambda)}$  \\
  \Xhline{1\arrayrulewidth}
  $\delta^x_{-1}$&{\begin{align*} \eta(x) & \rightarrow \eta(x) - 1 \\ \eta(x-1) & \rightarrow \eta(x-1) + 1 \end{align*}} &  $\frac{1}{2(1 + \lambda)}$ \\
  \\
    \Xhline{5\arrayrulewidth}
\end{tabular}
}
\end{center}
\caption{\label{tab:table-name}All values of the ARW toppling operator $\delta^x$ with corresponding effects on the configuration and probabilities of occurence.}
\end{table}

We update our configuration at time $t$ by choosing an unstable site $x_t$ and toppling it. We let the effects of the toppling move performed at $x_t$ be encoded in the (random) toppling operator $\delta_{x_t},$ so that $\eta_{t+1} = \delta_{x_t} \eta_t.$

\subsection{Sitewise Representation for \textbf{SS} and \textbf{ARW} on $\mathbb{Z}_n$}
We now introduce a framework for the two models which prescribes all of the randomness from the start, making the dynamics easier to study (see \citep{r8} for an early example of this formulation). For each model, we consider a \textit{random field of instructions}
\begin{align*}
\mathcal{I}_{SS} &= \{\tau_x^i \, : \, x \in \mathbb{Z}_n, \, i \in \mathbb{N}\} \\
\mathcal{I}_{ARW} &= \{\delta_x^i \, : \, x \in \mathbb{Z}_n, \, i \in \mathbb{N}\}.
\end{align*}
The set of instructions can be thought of as an infinite stack of instructions at each site in $\mathbb{Z}_n,$ and each time we topple a site $x$, we use the toppling operator at the top of the stack at $x$ and discard it. All of the randomness of the model is stored in the random set of instructions.

The following discussion applies to both SS and ARW. Let $\alpha = \{x_0,x_1,...\}$ represent the chronologically-ordered sequence of sites which are toppled, with length $T \in \mathbb{N} \cup \{\infty\}.$ We call $\alpha$ \textit{legal} if $x_t$ is unstable for all $t$. We call $\alpha$ $\textit{stabilizing}$ if the sequence of toppling moves results in the stable configuration. Define the odometer function $v_\alpha : \mathbb{Z}_n \rightarrow \mathbb{N}_0 \cup \{\infty\}$ to be the function which counts the number of times each site has toppled in the sequence $\alpha$. We now state the least action principle, which asserts that any legal sequence can be at most as long as a legal, stabilizing sequence.

\begin{lemma}[Least Action Principle, \citep{r6, r4}]
Let $\alpha$ and $\beta$ both be legal sequences for SS. If $\beta$ is stabilizing, then $v_\alpha(x) \leq v_\beta(x)$ for all $x$.

The same result holds for ARW odometers.
\end{lemma}

We next present the abelian property, which asserts that given a random set of instructions, the odometer function is independent of the choice of legal, stabilizing sequence. 

\begin{lemma} [Abelian Property, \citep{r6, r4}]
Fix a stack of instructions for SS, and let $\alpha$ and $\beta$ both be legal and stabilizing sequences. Then $v_\alpha = v_\beta.$

The same result holds for ARW odometers.
\end{lemma}

The abelian property asserts that the odometer function for a given stack of instructions is well-defined and independent of the choice of legal and stabilzing sequence. Accordingly, we henceforth drop the subscript and refer to the odometer functions (for a given stack of instructions) as $v$ and $u$ for SS and ARW, respectively. We define the stabilization time to be number of toppling moves required for stabilization:
\begin{align*}
T_{SS} = \sum_{x \in \mathbb{Z}_n} v(x) \\
T_{ARW} = \sum_{x \in \mathbb{Z}_n} u(x) 
\end{align*}
which by the abelian property is independent of the choice of legal, stabilizing sequence. We emphasize that the above stabilization times are functions of the random stack of instructions.

\subsection{Fast-Phase Stabilization for ARW}
We conclude this section by stating the fast-phase stabilization result from \citep{r7}: for sufficiently high sleep rate, $ARW$ on the $n$-cycle stabilizes in $O(n^3)$ time.

\begin{lemma} [Fast-Phase ARW Stabilization \citep{r7}]
Let $T_{\text{ARW}}$ be the (random) stabilization time for ARW with sleep rate $\lambda(n)$ on $\mathbb{Z}_n$ with arbitrary initial particle configuration. If
\[
\liminf_{n \rightarrow \infty} \frac{\lambda(n)}{\log(n)} > \frac{1}{2},
\]
then 
\[
\lim_{n \rightarrow \infty} \mathbbm{P} (T_{\text{ARW}}(n) > B n^3) = 0
\]
for some constant $B$ independent of $n$.
\end{lemma}

\section{Quotients of Markov Chains}
We now develop a framework by which we can couple SS to ARW. Our goal will be to compare ARW to SS by "projecting out" the particle states from the ARW Markov chain to obtain a Markov chain which is isomorphic to SS.

Let $H$ be a finite set, and consider the measurable space $(H,2^H)$. Let $S = \cup_{i = -1}^{|S|-2} H_i$ be a partition of $H$ into at least two sets, with none of the $H_i$ empty. Let the map $\pi : H \rightarrow S$ project an element of $H$ to its cell in $S$. We will consider the Markov chain $\{\eta_t\}_{t \in \mathbb{N}_0}$ taking values in $H$ with transition probability $p^H: H \times 2^H \rightarrow \mathbb{R}$, and which starts in a state $\eta_0$ such that $\pi(\eta_0) = H_0.$

\begin{definition} We call the partition $S = \{H_{-1},H_0,...\}$ \textbf{Markov-compatible with} $\{\eta_t\}$ \textbf{up to} $H_{-1}$ if
for all $H_i \in S \backslash H_{-1}$ and $H_j \in S,$ $p^H(\eta,H_j)$ is independent of the choice of  $\eta \in H_i$.
\end{definition}

\begin{definition} Let $\{\eta_t\}$ be a Markov chain on a finite set $H$ with initial state $\eta_0$, and let $S = \{H_{-1},H_0,...\}$ be a partition of $H$ which is Markov-compatible with $\{\eta_t\}$ up to $H_{-1}$. The \textbf{quotient transition probability} $p^H_S$ is defined as follows:
\[
p_S^H (H_i, H_j) = \begin{cases}
p^H(\eta, H_j) \text{ for any } \eta \in H_i & H_i \neq H_{-1} \\
\delta_{H_i,H_j} & H_i = H_{-1}.
\end{cases}
\]
Furthermore, let $\{s_t \}_{t \in \mathbb{N}_0}$ be a Markov chain taking values in $S$ with transition probability $p^H_S,$ and which starts in state $s_0 = H_0 \ni \eta_0.$ We refer to this as the \textbf{quotient Markov chain} of $\{\eta_t\}.$
\end{definition}
Note that the quotient transition probability is well-defined by the definition of Markov-compatibility, and is easily seen to be a bona-fide transition probability. Before stating our main Lemma from this section, we collect another definition:
\begin{definition}
Let $T_{-1}$ be the \textbf{hitting time for the set }$H_{-1}$:
\[
T_{-1} := \min\{t \geq 0 \, : \, \eta_t \in H_{-1}\},
\]
and let $\overline{\eta}_t := \eta_{\min (t,T_{-1})}$ be the corresponding \textbf{stopped Markov Chain}.
\end{definition}

We now make the following claim:
\begin{lemma} [Quotient Markov Chains] $(s_0,s_1,...)$ has the same law as $(\pi (\overline{\eta}_0), \pi (\overline{\eta}_1),...)$. \end{lemma}

\begin{proof} It suffices to show that:
\begin{equation}
p^H_S(H_i,H_j) = \mathbb{P} \big(\pi(\overline{\eta}_{t+1}) = H_j \, | \, \pi(\overline{\eta}_t) = H_i\big).
\end{equation}
First we assume that $H_i \neq H_{-1}.$ Consider the RHS of $(9).$ Since $\pi(\overline{\eta}_t) \neq H_{-1},$ we know that $t < T_{-1}.$ Thus the RHS of $(9)$ can be written as
\[
\mathbb{P} \big(\pi(\eta_{t+1}) = H_j \, | \, \pi(\eta_t) = H_i\big) = \mathbb{P} \big(\eta_{t+1} \in H_j \, | \, \eta_t \in H_i\big) = p^H(\eta \in H_i, H_j)
\]
where the last expression is well-defined over the choice of $\eta$ because $S$ is Markov-compatible up to $H_{-1}.$ Using the definition of $p_S^H$ for $H_i \neq H_{-1},$ we are done with the case of $t < T_{-1}.$

Now let $H_i = H_{-1}$ in the RHS of $(9).$ Since we are conditioning on $\pi(\overline{\eta}_t) = H_{-1},$ we have that $t \geq T_{-1}.$ Thus we can write
\[
\mathbb{P} \big(\pi(\overline{\eta}_{t+1}) = H_j \, | \, \pi(\overline{\eta}_t) = H_{-1}\big) = \mathbb{P} \big(\pi(\eta_{T_{-1}}) = H_j \, | \, \pi(\eta_{T_{-1}}) = H_{-1}\big) = \delta_{H_j, H_{-1}},
\]
which is equal to $p_S^H(H_{-1},H_j).$ This proves the Lemma.
\end{proof}

\section{SS as a quotient of ARW}
This section is devoted to showing that the stochastic sandpile model is isomorphic to a quotient Markov chain of activated random walk. We first define a (legal and stabilizing) toppling prescription for both models, so that each can be viewed as a well-defined Markov chain.

\subsection{A Toppling Prescription for ARW and SS}
We would now like to fix a legal, stabilizing toppling prescription for ARW, which we will write as a function of time: $j : \mathbb{N}_0 \rightarrow \mathbb{Z}_n$. Note that $j$ is also a function of the random field of instructions, though we leave this dependence implicit for ease of notation. We define $j(t)$ for even and odd times separately. For all $k \in \mathbb{N}_0$:
\begin{align}
j(2k) &= \begin{cases}
\min \{x \, : \, s_{2k}(x) \geq 2\} & \text{if } \max_{x \in \mathbb{Z}_n} s_{2k}(x) \geq 2 \\
\min \{x  \, : \, x \text{ has an active particle}\} & \text{else}
\end{cases} \\
j(2k + 1) &= \begin{cases}
j(2k) & \text{if } \max_{x \in \mathbb{Z}_n} s_{2k}(x) \geq 2 \\
\min \{x  \, : \, ix\text{ has an active particle}\} & \text{else}.
\end{cases}
\end{align}
In words: at even times, we topple the first site counterclockwise of the origin with $2$ or more particles. If we are unable to find such a site, we topple the first site counterclockwise of the origin with an active particle. Note that in either case, the toppling move is always legal. If the former case occurs, that is, if we've just toppled a site with $2$ or more particles, we topple this same site again at the next odd time. This toppling move is legal, since the particle which remains at the site (there is at least one) will still be active. If the latter case occured, that is, if we toppled a single active particle, we again topple the first active particle counterclockwise of the origin, if one exists. We stop toppling once we've stabilized, that is, once every particle is asleep. With the toppling prescription fixed, we can now think of ARW as a Markov chain. The state space $H$ of this Markov chain is the set of all ARW configurations, that is, the set of functions $\eta: \mathbb{Z}_n \rightarrow \mathbb{N}_{\mathfrak{s}}$ satisfying Definition 2. 

Our toppling prescription for SS will simply be to topple the first site counterclockwise of the origin, that is, we topple the site 
\begin{equation}
j(t) = \min\{x \, : \, s_t(x) \geq 2\}.
\end{equation}

\subsection{Coupling}

\begin{definition}
The \textbf{SS Markov Chain} $\{s_t\}_{t \in \mathbb{N}_0}$ takes values in the set $S$ of all SS configurations. We let one Markov chain step correspond to one SS toppling, according to the toppling prescription given in (12).

The \textbf{ARW Markov Chain} $\{\eta_t\}_{t \in \mathbb{N}_0}$ takes values in the set $H$ of all ARW configurations. We let one Markov chain step correspond to two ARW topplings, according to the toppling prescription given in (10) and (11).
\end{definition}

For ease of coupling SS with ARW, we let one step in the ARW Markov chain correspond to two ARW toppling moves. That is, if ARW stabilizes in $T_{ARW}$ toppling moves, we think of this as $\lceil \frac{T_{ARW}}{2} \rceil$ Markov Chain steps. Each SS toppling move will constitute one SS Markov chain step.

Consider the equivalence relation on the set of ARW configurations defined as follows:
\begin{equation}
\text{For } \eta_1, \eta_2 \in H, \eta_1 \sim \eta_2 \iff |\eta_1(x)| = |\eta_2(x)| \quad \forall x.
\end{equation}
It is clear that the partition given by $\sim$ can be identified with the set of all stochastic sandpile configurations $S$. Using notation from the previous section, we define $H_{-1} = (1,1,...,1)$. We now show the following:

\begin{lemma} The partition $S$ is Markov-compatible with the ARW Markov chain up to $H_{-1}.$ \end{lemma}

\begin{proof} Let $H_i \in S \backslash H_{-1}$, $\eta_1,\eta_2 \in H_i$, and $H_j \in S$ be artbirary. We need to show that $p^H (\eta_1,H_j) = p^H (\eta_2, H_j).$ Let $s_i, s_j$ be the stochastic sandpile configurations (ignoring states of particles) representing $H_i$ and $H_j,$ respectively. Since neither of $\eta_1, \eta_2$ are in $H_{-1},$ we topple (in the ARW sense) the clockwise-most site with two or more particles twice. In particular, the site to be toppled is the same for $\eta_1$ as it is for $\eta_2$. Call this site $x.$ 

We now state a pair of conditions which are both necessary for a state $\eta \in H_i$ to have non-zero transition probability to $H_j.$ Firstly, note that the numbers and states of particles at sites other than $x, x-1, x+1$ are unchanged by a toppling move, so that a set $H_j$ is accessible to $\eta \in H_i$ only if $s_j|_{\{x-1,x,x+1\}^c} = s_i|_{\{x-1,x,x+1\}^c}$. Second, make the following definitions:
\begin{align}
\rho_- &:= s_j (x-1) - s_i(x-1) \\
\rho_+ &:= s_j (x+1) - s_i(x+1)
\end{align}
where $x-1$ and $x+1$ are to be interpreted modulo $n$. $H_j$ is accessible from $\eta \in H_i$ only if $(\rho_-,\rho_+)$ satisfy equations (2) and (3), and $s_j(x) - s_i(x) = - (\rho_- + \rho_+).$ This follows from particle conservation and the fact that at most two particles can be displaced per toppling.

We now assume that $H_i$ and $H_j$ are such that both of the above conditions are met (if this is not the case, then $p^H(\eta,H_j) = 0$ for all $\eta \in H_i$). We enumerate all possible sets $H_j$ to which $\eta \in H_i$ can transition with positive probability by the pair $(\rho_-,\rho_+).$ In what follows, we argue that the transition from $H_i$ to $H_j$ is independent of which state in $H_i$ we move from. We discuss each possibility for $(\rho_-,\rho_+),$ case by case.\newline \textbf{Case 1: $(\rho_-, \rho_+) = (0,0)$}: This case holds if and only if neither particle is displaced from $x$, that is, that both particles being toppled attempted to go to sleep. This happens with probability $\big(\frac{\lambda}{1+\lambda}\big)^2,$ independently of the details of the state $\eta_1$ or $\eta_2$. Note that this case implies $H_j = H_i.$ \newline
\textbf{Case 2: $(\rho_-, \rho_+) = (1,0)$}: There are two possible toppling outcomes corresponding to this case: a) the first particle steps clockwise, while the second one tries to fall asleep, and b) the first particle tries to fall asleep, and the second one steps clockwise. Each of these two possibilities occurs with probability $\frac{1/2}{1+\lambda }\frac{\lambda}{1+\lambda},$ giving a total transition probability of $\frac{\lambda}{(1+\lambda)^2}$. 
\newline
\textbf{Case 3: $(\rho_-, \rho_+) = (0,1)$}: Both particles step counterclockwise. This is similar to Case 2.
\newline
\textbf{Case 4:  $(\rho_-, \rho_+) = (1,1)$}: One particle steps clockwise, and one particle steps counterclockwise. This happens with probability $2 \big( \frac{1/2}{1+\lambda}\big)^2$, where the factor of $2$ represents that the left and right steps can occur in any order. 
\newline
\textbf{Case 5:  $(\rho_-, \rho_+) = (2,0)$}: Both particles step clockwise. This occurs with probability $\big( \frac{1/2}{1+\lambda}\big)^2$.
\newline
\textbf{Case 6:  $(\rho_-, \rho_+) = (0,2)$}: Both particles step counterclockwise. This is similar to Case 5.

Since each transition probability is independent of the choice of ARW configuration in $H_i,$ we have shown that $S$ is Markov-compatible with the ARW Markov chain up to $H_{-1}.$ We record our non-zero transition probabilities in Table 3.
\end{proof}

\begin{table}
\begin{center}
{\renewcommand{\arraystretch}{1.5}
\begin{tabular}{ ? m{4.5cm}  | m{17em} | m{3.8cm} ? }
  \Xhline{5\arrayrulewidth}
\vspace{.1in} 
  Original State $\eta|_{\{x-1,x,x+1\}}$& $\Big( \eta(x-1),\eta(x),\eta(x+1)\Big)$ &  \\ 
  \Xhline{5\arrayrulewidth}
\vspace{.08in}
   $(\rho_-, \rho_+)$\vspace{.08in} & Transformed states $\eta'|_{\{x-1,x,x+1\}}$ & Transition Probabilities  \\ 
    \Xhline{5\arrayrulewidth}
  (0,0) &  $\Big( \eta(x-1),\eta(x),\eta(x+1)\Big)$ &  $\big(\frac{\lambda}{1+\lambda}\big)^2$ \\
    \Xhline{1\arrayrulewidth}
  (1,0) &  $\Big( \eta(x-1)+ 1,\eta(x) - 1,\eta(x+1)\Big)$ &  $\frac{\lambda}{1+\lambda} \frac{1/2}{1+\lambda}$ \\
   &  $\Big( \eta(x-1)+ 1,(\eta(x) - 1)\cdot \mathfrak{s},\eta(x+1)\Big)$ &  $\frac{1/2}{1+\lambda }\frac{\lambda}{1+\lambda} $ \\
    \Xhline{1\arrayrulewidth}
  (0,1) &  $\Big( \eta(x-1),\eta(x) - 1,\eta(x+1) + 1\Big)$ &  $\frac{\lambda}{1+\lambda} \frac{1/2}{1+\lambda}$ \\
   &  $\Big( \eta(x-1),(\eta(x) - 1)\cdot \mathfrak{s},\eta(x+1) + 1\Big)$ &  $\frac{1/2}{1+\lambda}\frac{\lambda}{1+\lambda} $ \\
    \Xhline{1\arrayrulewidth}
  (1,1) &  $\Big( \eta(x-1) + 1,\eta(x) - 2,\eta(x+1) + 1\Big)$ &  $2 \big( \frac{1/2}{1+\lambda}\big)^2$ \\
    \Xhline{1\arrayrulewidth}
   (2,0) &  $\Big( \eta(x-1) + 2,\eta(x) - 2,\eta(x+1)\Big)$ &  $\big( \frac{1/2}{1+\lambda}\big)^2$ \\
    \Xhline{1\arrayrulewidth}
   (0,2) &  $\Big( \eta(x-1),\eta(x) - 2,\eta(x+1)+2\Big)$ &  $\big( \frac{1/2}{1+\lambda}\big)^2$ \\
  \Xhline{5\arrayrulewidth}
\end{tabular}
}
\end{center}
\caption{\label{tab:table-name} Possible transitions of the ARW Markov Chain. Each pair $(\rho_-, \rho_+)$ represents an $H_j$ accessible from $\eta,$ with the middle column listing the accessible states $\eta' \in H_j$ from $\eta.$}
\end{table}

\begin{lemma}[Coupling Lemma]
Consider ARW with sleep rate $\lambda(n)$. For any initial configuration $\eta_0,$ the stopped, projected Markov Chain $\{\pi(\overline{\eta}_t\}_{t \in \mathbb{N}_0}$ on $S$ is isomorphic to the SS Markov chain $\{s_t\}_{t \in \mathbb{N}_0}$ on $S$, starting from $s_0 = \pi (\eta_0),$ with lazy parameter $p(n) = \frac{\lambda(n)}{1 + \lambda(n)}.$
\end{lemma}

\begin{proof} By Lemmas 4 and 5, it suffices to show that the SS Markov chain evolves according to the transition probability $p^H_S.$ This is easily checked by comparing corresponding rows in Tables 1 and 3. For any given pair $(\rho_-, \rho_+)$ which represents an accessible $H_j$ in the quotient chain, we obtain $p_S^H$ by summing the transition probabilities of all accessible states in $H_j$ in Table 3. Setting $p(n) = \frac{\lambda(n)}{1 + \lambda(n)},$ this gives transitions identical to those found in Table 1.
\end{proof}
 
\begin{proof}[Proof of Theorem 1]

Theorem 1 is a simple consequence of Lemma 6. Consider the coupled isomorphic Markov chains. For a fixed field of instructions, the number of times each site has SS-toppled is exactly two times the number of times each site has ARW-toppled up until the very last step, thanks to our compatbile toppling prescriptions defined in (10), (11), and (12). The last step can correspond to either one or two ARW topplings, giving the formula $v(x) = \lceil \frac{\overline{u}}{2} \rceil$. That $\overline{u}(x) \leq u(x)$ almost surely for all $x$ is trivial.
\end{proof}

%
%

\section*{Acknowledgements}
The author thanks Lionel Levine for introducing him to the problem and for his helpful discussions and ideas. The author is partially supported by grant DMS-1455272.


\begin{thebibliography}{4}

\bibitem{r1} Bak, P., Tang, C., Wiesenfeld, K. (1987). Self-organized criticality: An explanation of the 1/f noise. Physical Review Letters, 59(4), 381–384.

\bibitem{r2} Dhar, D. (1999). The abelian sandpile and related models. Physica A: Statistical Mechanics and Its Applications, 263(1-4), 4–25.

\bibitem{r3} Manna, S S. (1991). Two-state model of self-organized criticality. \textit{J. Phys. A: Math. Gen.} 24 L363

\bibitem{r6} Rolla, Leonardo T. (2020). Activated random walks on $\Bbb{Z}^d$. \textit{Probab. Surv.} 17, 478--544.

\bibitem{r9} Dhar, Deepak. (2006). Theoretical studies of self-organized criticality. \textit{Phys. A} 369, no. 1, 29--70.

\bibitem{r10} Ktitarev, D. V., Lübeck, S., Grassberger, P., B. Priezzhev, V. (2000). Scaling of waves in the bak-tang-wiesenfeld sandpile model. Physical Review E, 61(1), 81–92.

\bibitem{r4} Rolla, Leonardo T.; Sidoravicius, Vladas. (2012). Absorbing-state phase transition for driven-dissipative stochastic dynamics on ${\Bbb Z}$. \textit{Invent. Math.} 188, no. 1, 127--150.

\bibitem{r5} Basu, Riddhipratim; Ganguly, Shirshendu; Hoffman, Christopher; Richey, Jacob. (2019). Activated random walk on a cycle. \textit{Ann. Inst. Henri Poincaré Probab. Stat.} 55, no. 3, 1258--1277

\bibitem{r7} Hannah Cairns, Shirshendu Ganguly, and Lionel Levine. (2021). Phase transition for Activated Random Walk on a cycle, In preparation.

\bibitem{r8} Diaconis, P.; Fulton, W.  (1993). A growth model, a game, an algebra, Lagrange inversion, and characteristic classes. Commutative algebra and algebraic geometry, II (Italian) (Turin, 1990). \textit{Rend. Sem. Mat. Univ. Politec. Torino} 49 (1991), no. 1, 95--119.



\end{thebibliography}
\end{document}